\documentclass[reqno]{amsart}
\usepackage{}
\usepackage[top=1in, bottom=1in, left=1in, right=1in]{geometry}
\usepackage{cases}
\usepackage[colorlinks,
            linkcolor=red,
            anchorcolor=blue,
            citecolor=green
            ]{hyperref}

\newtheorem{theorem}{Theorem}[section]

\theoremstyle{definition}

\newtheorem{prop}[theorem]{Proposition}

\theoremstyle{remark}

\renewcommand{\dim}{\mbox{\rm dim}}

\numberwithin{equation}{subsection}
\numberwithin{theorem}{subsection}

\begin{document}\large

\title{On the elliptic curves
$ y^{2} = x ( x \pm p ) ( x \pm q ) $ over imaginary quadratic
number fields of class number one}
\author{Xiumei Li}
\address{Department of Mathematical Science, Tsinghua University, Beijing, P. R. China 100084}
\email{xm-li09@mails.tsinghua.edu.cn}






\begin{abstract} Let $ p $ and  $ q $ be odd prime numbers with
$ q - p = 2, $ the $\varphi -$Selmer groups, Shafarevich-Tate groups
( $ \varphi - $ and $ 2-$part ) and their dual ones as well the
Mordell-Weil groups of elliptic curves $ y^{2} = x ( x \pm p ) ( x
\pm q ) $ over imaginary quadratic number fields of class number one
are determined explicitly in many cases.\end{abstract}

\maketitle

\section{Introduction and Main Results} Let $ p $ and  $ q $ be odd prime numbers with $ q - p = 2. $ We
consider the elliptic curves  \begin{equation} \label{equation:basic1}E = E_{\varepsilon }: y^{2} = x ( x
+ \varepsilon p ) ( x + \varepsilon q ) \quad ( \varepsilon = \pm 1
) \end{equation} Write $ E = E_{+} $ if $
\varepsilon = 1 $ and $ E = E_{-} $ if $ \varepsilon = -1. $
Let the elliptic curves \begin{equation}\label{equation:basic2} E^{\prime } = E_{\varepsilon }^{\prime }:
y^{2} = x^{3} - 2 \varepsilon ( p + q ) x^{2} + 4 x.\end{equation} be the isogenous
curves, where the two-isogeny is $$ \varphi : E \longrightarrow E^{\prime}, \quad ( x, y ) \longmapsto
( y^{2} / x^{2}, \ y( pq - x^{2} ) / x^{2}) $$ with ker$(\varphi)= E [
\varphi ] = \{ O, \ (0, 0) \}$ and  $$ \widehat{ \varphi } : E^{\prime } \longrightarrow E, \quad (
x, y ) \longmapsto ( y^{2} /4 x^{2}, \ y( 4 - x^{2} ) /8 x^{2}) $$ is the dual isogeny of $ \varphi $
with kernel $ E^{\prime} [ \widehat{ \varphi } ] = \{ O, \ (0, 0) \}.
$ \\
\indent D.Qiu \cite{Qiu1} gave out many results about the selmer group,Shafarevich-Tate groups and Mordell-Weil groups of $E$ over $\mathbb{Q}.$ In this paper, we mainly generalized theorem 1[5] and
theorem 2 \cite{Qiu1} to imaginary quadratic fields $K$ with class number one.
Gauss-Baker-Stark theorem \cite{Zhang} tell us that there are exactly nine such fields $ K = \mathbb{Q}
(\sqrt{D} ) $ with fundamental discriminant $ D $ given by \\
\indent \indent $  D = -3, -4, -7, -8, -11, -19, -43, -67, -163. $ \\
\indent Our main purpose in this paper is to determine the Selmer groups $
S^{(\varphi )} ( E / K ), \ S^{(\widehat{\varphi })} ( E^{\prime } /
K ), $ \ Shafarevich-Tate groups $ \text{TS}( E / K )[2], \
\text{TS}( E / K )[\varphi ], \ \text{TS}( E^{\prime } / K
)[\widehat{\varphi }], $ \ Mordell- Weil group $ E ( K ) $ and $
\text{rank} ( E ( K ) ). $ There are many
literature studying special types of elliptic curves by using $
2-$descent method ( see e.g., \cite{Bremner1}, \cite{Bremner2}, \cite{Dabrowski}, \cite{Strocker}).
\par \vskip 0.1 cm

\begin{theorem}\label{theorem:main1}  Let $ E = E_{\varepsilon } $ and $
E^{\prime } = E_{\varepsilon }^{\prime } $ be the
elliptic curves in (1) and (2) with $ \varepsilon = \pm 1. $ \begin{itemize}
\item[(A)] Assume that condition (A) holds, then\\
 $$ S^{(\varphi )} (E_{+} / K ) \cong \  \left \{
   \begin{array}{l}
  0 \qquad \quad \text{if} \ p \equiv 3, 17 (\bmod 56 ), \\
  ( \mathbb{Z} / 2 \mathbb{Z} ) \quad \text{if} \ p \equiv 45 (\bmod 56
   ), \\
   \left( \mathbb{Z} / 2 \mathbb{Z} \right)^{2} \quad \text{if} \ p \equiv 31 (\bmod 56
   ).
  \end{array}
  \right.    $$
  $$ S^{(\widehat{\varphi } )} (E^{\prime }_{+} / K ) \cong \  \left \{
   \begin{array}{l}
  \left( \mathbb{Z} / 2 \mathbb{Z} \right)^{2} \quad \text{if} \ p \equiv 45 (\bmod 56
   ), \\
   \left( \mathbb{Z} / 2 \mathbb{Z} \right)^{3} \quad \text{if} \ p \equiv 3, 17, 31 (\bmod 56
   ).
  \end{array}
  \right.    $$
\ $$ S^{(\varphi )} (E_{-} / K ) \cong \  \left \{
   \begin{array}{l}
  0 \qquad \quad \text{if} \ p \equiv 3, 45 (\bmod 56 ), \\
  ( \mathbb{Z} / 2 \mathbb{Z} ) \quad \text{if} \ p \equiv 17 (\bmod 56
   ), \\
   \left( \mathbb{Z} / 2 \mathbb{Z} \right)^{2} \quad \text{if} \ p \equiv 31 (\bmod 56
   ).
  \end{array}
  \right.   $$
  $$ S^{(\widehat{\varphi } )} (E^{\prime }_{-} / K ) \cong \  \left \{
   \begin{array}{l}
  \left( \mathbb{Z} / 2 \mathbb{Z} \right)^{2} \quad \text{if} \ p \equiv 17 (\bmod 56
   ), \\
   \left( \mathbb{Z} / 2 \mathbb{Z} \right)^{3} \quad \text{if} \ p \equiv 3, 31, 45
   (\bmod 56 ).
  \end{array}
  \right. $$
\item[(B)] Assume that condition (B) holds, then
$$ S^{(\varphi )} (E / K ) \cong \ \mathbb{Z} / 2 \mathbb{Z}, \qquad
S^{(\widehat{\varphi } )} (E^{\prime } / K ) \cong \ \left( \mathbb{Z} / 2
\mathbb{Z} \right)^{3}. $$
\item[(C)] Assume that condition (C) holds, then
$$ S^{(\varphi )} (E_{+} / K ) \cong \ \{ 0 \}, \qquad
S^{(\widehat{\varphi } )} (E^{\prime }_{+} / K ) \cong \ \left( \mathbb{Z} / 2
\mathbb{Z} \right)^{3}. $$
$$ S^{(\varphi )} (E_{-} / K ) \cong \ \mathbb{Z} / 2 \mathbb{Z}, \qquad
S^{(\widehat{\varphi } )} (E^{\prime }_{-} / K ) \cong \ \left( \mathbb{Z} / 2
\mathbb{Z} \right)^{3}. $$
\item[(D)]  Assume that condition (D) holds, then
 \ $$ S^{(\varphi )} (E / K ) \cong \  \left \{
   \begin{array}{l}
  0 \qquad \quad \text{if} \ p \equiv 3, 5 (\bmod 8 ), \\
  ( \mathbb{Z} / 2 \mathbb{Z} ) \quad \text{if} \ p \equiv 1, 7 (\bmod 8
   ). \\
   \end{array}
  \right.    $$
  $$ S^{(\widehat{\varphi } )} (E^{\prime } / K ) \cong \  \left \{
   \begin{array}{l}
  \left( \mathbb{Z} / 2 \mathbb{Z} \right)^{2} \quad \text{if} \ p \equiv 5 (\bmod
  8 ), \\
   \left( \mathbb{Z} / 2 \mathbb{Z} \right)^{3} \quad \text{if} \ p \equiv 1, 3 (\bmod
   8 ), \\
   \left( \mathbb{Z} / 2 \mathbb{Z} \right)^{4} \quad \text{if} \ p \equiv 7 (\bmod
   8 ).
  \end{array}
  \right.    $$
\item[(E)] Assume that condition (E) holds, then \\
 \ $$ S^{(\varphi )} (E / K ) \cong \  \left \{
   \begin{array}{l}
  0 \qquad \quad \text{if} \ p \equiv 5, 11 (\bmod 24 ), \\
  ( \mathbb{Z} / 2 \mathbb{Z} ) \quad \text{if} \ p \equiv 17, 23 (\bmod 24
   ). \\
   \end{array}
  \right.    $$
  $$ S^{(\widehat{\varphi } )} (E^{\prime } / K ) \cong \  \left \{
   \begin{array}{l}
  \left( \mathbb{Z} / 2 \mathbb{Z} \right)^{3} \quad \text{if} \ p \equiv 5, 11 (\bmod
  24 ), \\
   \left( \mathbb{Z} / 2 \mathbb{Z} \right)^{4} \quad \text{if} \ p \equiv 17, 23 (\bmod
   24 ).
  \end{array}
  \right.    $$  \end{itemize}\end{theorem}
\indent For simplicity, we denote the dimension $ \dim_{2} V
= \dim_{\mathbb{F}_{2}} V $  for a vector space $ V $ over the field
$ \mathbb{F}_{2} $ of two elements.
\begin{theorem}\label{theorem:main2} Let $ E = E_{\varepsilon } $ and $
E^{\prime } = E_{\varepsilon }^{\prime } $ be the
elliptic curves in (1) and (2) with $ \varepsilon = \pm 1. $ \begin{itemize}
\item[(A$_{+}$)] Assume that condition (A) holds. For $ \varepsilon = 1, $ we have \\
(1)
$ \text{rank} ( E ( K ) ) + \text{dim}_{2} ( \text{TS}( E / K
)[\varphi ]) + \text{dim}_{2} ( \text{TS}( E^{\prime } / K
)[\widehat{\varphi } ]) = 3,$ if $ p \equiv 31 (\bmod 56 ); $ \\
(2) $ \text{rank} ( E ( K ) ) + \text{dim}_{2} ( \text{TS}( E^{\prime
} / K )[2 ]) = 1, $ if $ p \equiv 3, 17 (\bmod 56 );$ \\
(3) $ \text{rank} ( E ( K )) + \text{dim}_{2} ( \text{TS}( E / K )[2 ])
= 1,$ if $ p \equiv 45 (\bmod 56 ). $
 \item[(A$_{-}$)] Assume that condition (A) holds. For $ \varepsilon = - 1, $ we have \\
(1)
$ \text{rank} ( E ( K ) ) + \text{dim}_{2} ( \text{TS}( E^{\prime
} / K )[2]) = 1, $ if $ p \equiv 3, 45 (\bmod 56 ); $  \\
(2) $ \text{rank} ( E ( K ) ) + \text{dim}_{2} ( \text{TS}( E / K )[2
]) = 1, $ if $ p \equiv 17 (\bmod 56 ); $ \\
(3)
$ \text{rank} ( E ( K ) ) + \text{dim}_{2} ( \text{TS}( E / K
)[\varphi ]) + \text{dim}_{2} ( \text{TS}( E^{\prime } / K
)[\widehat{\varphi }]) = 3, $ if $ p \equiv 31 (\bmod 56 ). $
\item[(B)] Assume that condition (B) holds, then
$$ \text{rank} ( E ( K ) ) + \text{dim}_{2} ( \text{TS}( E / K
)[\varphi ]) + \text{dim}_{2} ( \text{TS}( E^{\prime } / K
)[\widehat{\varphi } ]) = 2. $$
\item[(C)] Assume that condition (B) holds, then
$$  \quad \text{rank} ( E_{+} ( K ) ) +
\text{dim}_{2} ( \text{TS}( E^{\prime }_{+} / K )[2]) = 1. $$
$$ \text{rank} ( E_{-} ( K ) ) + \text{dim}_{2} ( \text{TS}( E_{-} / K
)[\varphi ]) + \text{dim}_{2} ( \text{TS}( E^{\prime }_{-} / K
)[\widehat{\varphi } ]) = 2. $$
\item[(D)]  Assume that condition (D) holds, then
 (1) $ \text{rank} ( E ( K ) ) + \text{dim}_{2} ( \text{TS}( E /
K )[\varphi ]) + \text{dim}_{2} ( \text{TS}( E^{\prime } / K )
[\widehat{\varphi } ]) = 2, $ if $ p \equiv 1 (\bmod 8 ); $ \\
(2) $ \text{rank} ( E ( K ) ) + \text{dim}_{2} ( \text{TS}( E^{\prime
} / K )[2 ]) = 1, $ if $ p \equiv 3 (\bmod 8 ); $ \\
(3) $ \text{rank} ( E ( K ) ) + \text{dim}_{2} ( \text{TS}( E /
K )[\varphi ]) + \text{dim}_{2} ( \text{TS}( E^{\prime } / K )
[\widehat{\varphi } ]) = 3, $ if $ p \equiv 7 (\bmod 8 ); $ \\
(4)  $$ \text{TS}( E / K )[2]
= 0, \quad \text{TS}( E^{\prime } / K )[2 ] = 0,   E ( K ) \cong \mathbb{Z} / 2 \mathbb{Z}
\times \mathbb{Z} / 2 \mathbb{Z}, $$ if $ p \equiv 5 (\bmod 8 ). $
\item[(E)]  Assume that condition (E) holds, then \\
(1)  $ \text{rank} ( E ( K ) ) + \text{dim}_{2} ( \text{TS}( E^{\prime
} / K )[2]) = 1, $ \ if $ p \equiv 5, 11 (\bmod 24 ); $ \\
(2)  $ \text{rank} ( E ( K ) ) + \text{dim}_{2} ( \text{TS}( E / K
)[\varphi ]) + \text{dim}_{2} ( \text{TS}( E^{\prime
} / K )[\widehat{\varphi }]) = 3, $ if $ p \equiv 17, 23 (\bmod 24 ). $ \end{itemize} \end{theorem}

\section{Computation of the Selmer groups}
Let $ M_{K} $ be the
set of all places of $ K. $ For
each place $ v\in M_{K}, $ let $ K_{v} $ be the completion of $ K $ at $
v, $ and $ \text{ord}_{v}( ) $ be the
corresponding normalized additive valuation, if $ v $ is finite.
 If $ \pi $ is an
irreducible element corresponding to $ v, $ then we simply denote $
\text{ord}_{v}( ) $ by $ v_{\pi }( ), $ so $ v_{\pi }(\pi ) = 1. $
Put $ S = \{ \infty \} \cup \{ \text{primes in } \ K \
\text{dividing} \  2 p q \}, $ and  $$ K(S, 2)= \{ d \in K^{\ast } / K^{\ast ^{2}}
: \ \text{ord}_{v}(d) \equiv 0 (\bmod 2) \ \text{ for all } \ v
\notin S \}. $$  For each $ d \in K(S, 2), $ define the curves
$$ C_{d} : \ d w^{2} = d^{2} - 2 \varepsilon ( p + q) d z^{2} + 4
z^{4}, $$
$$ C_{d}^{\prime} : \ d w^{2} = d^{2} +  \varepsilon ( p + q) d
z^{2} + p q z^{4}. $$
According to the algorithm in \cite{Silverman} chap. X, we have the following identifications:
$$ S^{(\varphi )} (E / K ) \cong \{ d \in K(S, 2) : \ C_{d} (
K_{v}) \neq \emptyset  \text{ for all } v \in S \}, $$
$$ S^{(\widehat{\varphi })} (E^{\prime } / K ) \cong \{ d \in K(S,
2) : \ C_{d}^{\prime } ( K_{v}) \neq \emptyset  \text{ for all } v
\in S \}. $$
\par  \vskip 0.1 cm

Next we divide our discussion according to $ K $ into the
following cases:
\par  \vskip 0.2 cm

\subsection{ Case A \ $ K = \mathbb{Q}(\sqrt{-7}) $ }

We denote the condition: "$K=\mathbb{Q} (\sqrt{-7}) $ and both $ p $ and $ q $ are inertia in $
K $ " by condition (A). In this section, we always assume that condition (A) holds.\\ \indent By ramification theory, condition (A) holds if and only if $ \left( \frac{- 7}{p}\right) = \left( \frac{- 7}{q}\right)
= - 1, $ it's also equivalent to that \ $ p \equiv 3, 17, 31, 45
(\bmod 56 ). $ Note that $ 2 $ splits completely in $ K, $
denote $ \pi _{2} = - \frac{ 1 + \sqrt{ - 7}}{2}, \ \overline{\pi
_{2}} = \frac{ - 1 + \sqrt{ - 7}}{2}. $
Under the above assumption and notation, here $ S = \{ \infty ,
\pi_{2} , \overline{\pi_{2} }, p, q \} $ and $ K( S, 2) = <-1, \pi
_{2}, \overline{\pi _{2}}, p,  q >. $ The completions $ K_{v} $ at $
v \in
S $ are given respectively by
$$ K_{\infty } = \mathbb{C}, \ K_{\pi _{2} } \cong K_{\overline{\pi _{2}} }
\cong \mathbb{Q}_{2}, \ K_{p} = \mathbb{Q}_{p} (\sqrt{-7}), \ K_{q} = \mathbb{Q}_{q}
(\sqrt{-7}). $$
For each $ v \in S \setminus \{\infty \}, $ we fix an embedding $
K \hookrightarrow K_{v} $ \ such that $ \text{ord}_{v}(v) = 1, $
taking $ K_{\pi _{2} } ( \supset K ), $ for an example, we have $ v_{\pi _{2} }
(\pi _{2}) = v_{\pi _{2} } (2) = 1 $ and $ v_{\pi _{2} }(
\overline{\pi _{2} }) = 0. $
\par  \vskip 0.2 cm
\  For  $ E = E_{+}$ be as in \ref{equation:basic1}, we have the following results:
\begin{prop}\label{prop:mod1}\begin{itemize}
\item[(1)] For $ d \in K(S, 2), $ if one of the following conditions holds: \\
(a) \ $ p \mid d; $ \quad  (b) \ $ q \mid d; $ \quad (c) \ $ d \in
\{ -1, - \pi _{2}, - \overline{\pi _{2}} \}. $ \quad  Then $ d
\notin S^{(\varphi )} (E / K ). $
\item[(2)](a) $ 2 \in S^{(\varphi )} (E / K ) \ \Longleftrightarrow
\ p \equiv 31 (\bmod 56 ); $ \\
(b) \ $ - 2 \in S^{(\varphi )} (E / K ) \
 \Longleftrightarrow \ p \equiv 45 (\bmod 56 ). $
\item[(3)](a) \ $ \pi _{2} \in S^{(\varphi )} (E / K ) \
\Longleftrightarrow \ p \equiv 31 (\bmod 56 );  $ \\
(b) \ $ \overline{\pi _{2}} \in S^{(\varphi )} (E / K ) \
\Longleftrightarrow \ p \equiv 31 (\bmod 56 ). $
\end{itemize}\end{prop}
\begin{proof}\begin{itemize} \item[(1)] follows directly by  valuation.
\item[(2)](a)  follows directly by valuation
and prop. 2.1 in \cite{Qiu1}.
 \item[(2)](b) let $ f(z, w) = w^{2} + 2 + 2 ( p + q ) z^{2} + 2 z^{4}, $ then
$ C_{-2}: f(z, w) = 0. $ \\
\indent i) To prove that $ C_{-2}( \mathbb{Q}_{2} ) \neq \emptyset$ if and only if $p \equiv 5 (\bmod \ 8 ).$ For necessity, first note, that $ C_{-2}( \mathbb{Q}_{2} ) \neq \emptyset $ implies $
C_{-2}(\mathbb{Z}_{2}) \neq \emptyset . $ Indeed, $ (z, w ) \in
C_{-2}(\mathbb{Q}_{2})$ implies $(1/z, w/z^{2} ) \in C_{-2}(
\mathbb{Q}_{2} ).$ Taking $(z, w ) \in C_{-2}(\mathbb{Z}_{2}
), $ by valuation property, we have $ w = 2 w_{0}, z = 1 + 2
z_{0} $ for $ z_{0}, w_{0} \in \mathbb{Z}_{2} $ and satisfying $$w_{0}^{2} = -2 ( 1 + 2 z_{0} + 2
z_{0}^{2})^{2} - p (1 + 2 z_{0} )^{2} $$ Taking the valuation $ v_{2}$ of both side and
by \cite{Robert} p.50, we obtain $ p \equiv 5 (\bmod \ 8 ). $ Conversely, Let $
g(z, w ) = (z^{2} + 1)^{2} + 2 p z^{2} + 2 w^{2}, $ by the above discussion, we know
that $ f(z, w )=0 $ has solutions in $ \mathbb{Q}_{2}^{2} $ if and only if
 $ g(z, w)=0 $ has solutions in $ \mathbb{Z}_{2}^{2}. $ Firstly, if
$ p \equiv 5 (\bmod \ 16 ), $ then $ t^{2} - 17 = 0 $ has a solution
 $ w_{0} $ in
$ \mathbb{Z}_{2}$(by Hesel lemma \cite{Silverman} p.322, by $ v_{2}(g(3, w_{0} ) )
> 2 v_{2}(g^{\prime }_{w}(3, w_{0} ) ) $ and Hensel lemma again,
$ g(z, w )=0 $ has solutions in $ \mathbb{Z}_{2}^{2}. $ Secondly, if $ p
\equiv 13 (\bmod \ 16 ), $ then by $ v_{2}(g(17, 17 ) ) > 2
v_{2}(g_{w}(17, 17 ) ) $ and Hensel lemma, $ g(z, w)=0 $ has
solutions in $ \mathbb{Z}_{2}^{2}. $ This proves $ C_{-2}( \mathbb{Q}_{2} ) \neq
\emptyset \Longleftrightarrow  p
\equiv 5 (\bmod \ 8 ). $ \\
\indent ii) To prove that if $ p \equiv 5 (\bmod \ 8 ), $then $ C_{-2}(K_{p}) \neq \emptyset.$ In fact,
if $ p \equiv 5 (\bmod \ 8 ), $ then $2$ is quadratic nonresidue modulo $\emph{p}.$ Combining with condition (A), we have the congruence $ 7 c^{2} \equiv 2 (\bmod \ p )$ for some $ c
\in \mathbb{Z}. $ By $ v_{p}(f(p, \sqrt{-7}c ) ) \geq 2 v_{p}(f_{w}(p,
\sqrt{-7}c ) ) = 0
$ and Hensel lemma, we see that $ C_{-2}(K_{p}) \neq \emptyset. $ \\
\indent iii) To prove that if $ p \equiv 5 (\bmod \ 8 ), $then $ C_{-2}(K_{q}) \neq \emptyset.$\ its proof is similarly to ii).\\
 Let us summarize the calculation:$ 2 \in S^{(\varphi )} (E / K ) \ \Longleftrightarrow \ p
 \equiv 31 (\bmod 56 ).$
\item[(3)](a) Let $ f(z, w) = w^{2} - \pi _{2} + 2 (p + q ) z^{2} -
2 \bar{\pi _{2}} z^{4}.$ \\
\indent i) To prove that $
C_{\pi _{2}}(K_{\bar{\pi _{2}}}) \neq \emptyset. $ By $v_{\bar{\pi _{2}}}(f(1, \pi _{2} ) ) \geq 3 > 2 = 2 v_{\bar{\pi
_{2}}}(f^{\prime }_{w}(1, \pi {2} ) )$ and Hensel lemma, $
C_{\pi _{2}}(K_{\bar{\pi _{2}}}) \neq \emptyset. $ \\
\indent ii) To prove that $ C_{\pi _{2}}(K_{\pi _{2}}) \neq \emptyset $ if and only if $p \equiv 31 (\bmod \ 56 ).$ For necessity, if $ C_{\pi _{2}}(K_{\pi _{2}}) \neq \emptyset, $ taking $ (z, w
) \in C_{\pi_{2}}(K_{\pi_{2}}), $  by valuation, $
v_{\pi _{2}}(z) = 0. $ As $ K_{\pi _{2}} = \mathbb{Q}_{2}, \ v_{\pi_{2}} =
v_{2}, $ we have $ v_{\pi _{2}}(z^{2} - 1 ) \geq 3 $ ( See[7,
p.50]), hence $ z^{2} - 1 = \pi _{2}^{3} z_{0} $ with $ z_{0} \in
\mathbb{Z}_{2}. $ Substituting them into the equation $f(z,w)=0,$
we get
\begin{equation} w^{2} = (\pi _{2} + 2 \bar{\pi _{2}}) - 4 (p + 1 ) - 4 (p + 1 )
\pi _{2}^{3} z_{0} + 8 \pi _{2}^{2} z_{0} + 4 \pi _{2}^{5}
z_{0}^{2}\end{equation}
Taking the valuation $ v_{\pi _{2}} $ of both sides, $ v_{\pi
_{2}}(w) = 1 $(note that $ \pi _{2} + 2 \bar{\pi _{2}} = \pi _{2}^{2}$ ),
  we can take $ w = \pi _{2} \cdot w_{0} $ for some $ w_{0} \in
\mathbb{Z}_{2}^{\ast }. $ Substituting into (3), we obtain $$ w_{0}^{2} =
1 - \bar{\pi _{2}}^{2}(p + 1 ) - 4 (p + 1 ) \pi _{2} z_{0} + 8 z_{0}
+ 4 \pi _{2}^{3} z_{0}^{2}. $$  Since $ v_{\pi _{2}}(w_{0}^{2} - 1 )
\geq 3, $ we get $ v_{\pi _{2}}(\bar{\pi _{2}}^{2}(p + 1 ) ) \geq 3,
$ so $ p + 1 \equiv 0 (\bmod \ 8 ), $ i.e., $ p \equiv 7
(\bmod \ 8 ). $ By condition (A), we obtain
$ p \equiv 31 (\bmod \ 56 ). $
Conversely, if $ p \equiv 31 (\bmod \ 56 ), $ then $ v_{2}(p
+ 1 ) \geq 3. $ By $
v_{\pi _{2}}(f(1, \pi _{2} ) ) > 2 v_{\pi _{2}}(f^{\prime }_{w}(1,
\pi _{2})) $ and Hensel lemma, $ C_{\pi _{2}}(K_{\pi _{2}}) \neq
\emptyset. $ \\
\indent iii) To prove that if $ p \equiv 17, 31 (\bmod \ 56), $ then $ C_{\pi _{2}}(K_{p}) \neq \emptyset. $ In fact, if $ p \equiv 17, 31 (\bmod \ 56), $ then $ (\frac{2}
{p} ) = 1, $ we can write $ a^{2} = 2 + p u $ for some $ a, u
\in \mathbb{Z},$ then $ 1 - 4 a^{2} \equiv -7 (\bmod \ p ).$ By condition (A), $ (\frac{-7} {p}) = -1, $  we have
 $ (\frac{-1 - 2a}{p}) (\frac{-1 + 2 a^{2}} {p}) = -1.$  Without loss of generality, we may
assume that $ (\frac{-1 - 2a}{p} ) = 1, $  then $ -1 - 2 a = b^{2}
- p v $ for some $ b, v \in Z. $ Taking $ \alpha = \frac{b}{2} (1 +
\frac{\sqrt{-7}}{1 + 2 a}) \in \mathbb{O}_{K_{P}} \subset K_{p} $ (note that $ p
\not| 2 b(1 + 2 a ) $ ), by $
v_{p}(f(0, \alpha ) ) > 2 v_{p}(f^{\prime }_{w}(0, \alpha )). $
and Hensel lemma, $ C_{\pi _{2}}(K_{p}) \neq \emptyset. $ \\
\indent iv) To prove that if $ p \equiv 17, 31 (\bmod \ 56), $ then $ C_{\pi _{2}}(K_{q}) \neq \emptyset. $ Its proof is similar to iii).\\
 Let us summarize the calculation:$ \pi _{2} \in S^{(\varphi )} (E / K ) \
\Longleftrightarrow \ p \equiv 31 (\bmod 56 ). $ \\
\indent (b) is similar to (a).\end{itemize}\end{proof}
\par  \vskip 0.1 cm

  For  $E^{\prime} = E^{\prime}_{+} $ be as \ref{equation:basic2}, we have the following results:
\begin{prop}\label{prop:mainA12}\begin{itemize}
\item[(1)] For $ d \in K(S, 2), $ if one of the following conditions holds: \\
(a) \ $ \pi _{2} \mid d; $ \quad  (b) \ $ \overline{\pi _{2}} \mid
d. $  \quad Then $ d \notin S^{(\widehat{\varphi} )} (E^{\prime } /
K ). $
\item[(2)] $ -p, \ -q \in S^{(\widehat{\varphi} )} (E^{\prime } / K ).
$
\item[(3)] $ -1 \in S^{(\widehat{\varphi} )} (E^{\prime } / K )
\Longleftrightarrow \ p \equiv 3, 17, 31 (\bmod 56 ). $\end{itemize}\end{prop}
\begin{proof}\begin{itemize} \item[(1)] is similar to Proposition 2.1A$_{+}$ (1).
\item[(2)] By the equation of $ C_{-p}^{'}, $ it is easy to see that
 $ (1, 0 ) \in C_{-p}^{'}(\mathbb{Q} ) \subset C_{-p}^{'}(K ) $ and $ (1, 0) \in
 C_{-q}^{'}(\mathbb{Q} ) \subset C_{-q}^{'}(K ). $ So $ -p, -q \in  E(K) /
\widehat{\varphi } ( E^{\prime } / K ) \subset S^{(\widehat{\varphi
})}(E^{\prime } / K ). $
\item[(3)] For necessity, if $ -1 \in S^{(\widehat{\varphi })}(E^{'} / K ), $ then $
C_{-1}^{\prime }(K_{\pi _{2}} ) \neq \emptyset , \ C_{-1}^{\prime
}(K_{\bar{\pi _{2}}} ) \neq \emptyset . $ Since $ K_{\pi _{2}} \cong
K_{\bar{\pi _{2}}} \cong \mathbb{Q}_{2}, $ we have $ C_{-1}^{\prime
}(\mathbb{Q}_{2}) \neq \emptyset. $ So by proposition 2.2 in [5] we obtain $
p \equiv 1, 3, 7(\bmod \ 8). $ By condition (A),  we get $ p \equiv
3, 17, 31 (\bmod \ 56 ). $
Conversely, if $ p \equiv 3, 17, 31 (\bmod \ 56), $ then by
prop. 2.2 in \cite{Qiu1}, we have $ C_{-1}^{\prime }(\mathbb{Q}_{2} ) \neq
\emptyset , \ C_{-1}^{\prime }(\mathbb{Q}_{p}) \neq \emptyset, $ and $
C_{-1}^{\prime }(\mathbb{Q}_{q}) \neq \emptyset . $ Since $ K_{\pi _{2}}
\cong K_{\bar{\pi _{2}}} \cong \mathbb{Q}_{2}, \ \mathbb{Q}_{p} \subset
 K_{p}, \ \mathbb{Q}_{q} \subset K_{q}, $ by definition, we obtain
$ -1 \in S^{(\widehat{\varphi })}(E^{'} / K ). $ \end{itemize}\end{proof}
\par  \vskip 0.1 cm

 Similarly, for  $ E = E_{-},E^{\prime} = E^{\prime}_{-}$ be as in (1) and (2),we have some similar results:
\begin{prop}\label{prop:mainA21}\begin{itemize}
\item[(1)] For $ d \in K(S, 2), $ if one of the following conditions holds: \\
(a) \ $ p \mid d; $ \quad  (b) \ $ q \mid d; $ \quad (c) \ $ d \in
\{ -1, - \pi _{2}, - \overline{\pi _{2}} \}. $ \ Then $ d \notin
S^{(\varphi )} (E / K ). $
\item[(2)] $ 2 \in S^{(\varphi )} (E / K ) \ \Longleftrightarrow \ p
 \equiv 31 (\bmod 56 ); $ \\
 $ - 2 \in S^{(\varphi )} (E / K ) \
 \Longleftrightarrow \ p \equiv 17 (\bmod 56 ). $
\item[(3)] $ \pi _{2} \in S^{(\varphi )} (E / K ) \ \Longleftrightarrow
\ p \equiv 31 (\bmod 56 );  $ \\
$ \overline{\pi _{2}} \in S^{(\varphi )} (E / K ) \
\Longleftrightarrow \ p \equiv 31 (\bmod 56 ). $\end{itemize}\end{prop}
\begin{prop}\label{prop:mainA22}\begin{itemize}
\item[(1)] For $ d \in K(S, 2), $ if one of the following conditions holds: \\
(a) \ $ \pi _{2} \mid d; $ \quad  (b) \ $ \overline{\pi _{2}} \mid
d. $  \quad Then $ d \notin S^{(\widehat{\varphi} )} (E^{\prime } /
K ). $
\item[(2)] $ p, \ q \in S^{(\widehat{\varphi} )} (E^{\prime } / K ). $
\item[(3)] \ $ -1 \in S^{(\widehat{\varphi} )} (E^{\prime } / K )
\Longleftrightarrow \ p \equiv 3, 31, 45 (\bmod 56 ). $ \end{itemize}\end{prop}
\par  \vskip 0.2 cm

\subsection{Case B \ $ K = \mathbb{Q} (\sqrt{D}) $ with $ D = -11, -19, -43, -67,
-163 $ }
We denote the condition: "$ K = \mathbb{Q} (\sqrt{D}) $ with $ D = -11, -19, -43, -67,
-163 $ and both $ p $ and $ q $ are inertia in $
K $ " by condition (B). In this section, we always assume that condition (B) holds.\\
\indent By ramification theory, condition (B) holds if and only if
 $ \left( \frac{D}{p}\right) = \left( \frac{D}{q}\right)
= - 1. $ Note that $ D \equiv 5 (\bmod \ 8 ), $ so $ 2 $ is
inertia in $ K $ and the corresponding residual field is $
\mathbb{O}_{K} /  2\mathbb{O}_{K}  \cong \mathbb{F}_{4}, $ the field
of four elements. One can take $ T = \{ 0, 1, \pi = \frac{ 1 +
\sqrt{D} }{ 2 }, - \overline{\pi } \} $ as the set of the
representatives of $ \mathbb{O}_{K} / 2\mathbb{O}_{K}. $  Under the above assumption and notation, here $ S = \{ \infty , 2 , p, q \}$ and $ K( S, 2) = <-1,  2,  p, q
>. $ The completions $ K_{v} $ at $ v \in S $ are given
respectively by
$$ K_{\infty } = \mathbb{C}, \ K_{2 } \cong \mathbb{Q}_{2} (\sqrt{D}), \ K_{p} =
\mathbb{Q}_{p} (\sqrt{D}), \ K_{q} = \mathbb{Q}_{q} (\sqrt{D}). $$
\par  \vskip 0.1 cm
\indent For  $ E
= E_{\varepsilon },E^{\prime }=E^{\prime }_{\varepsilon } $ with $ \varepsilon = \pm 1 $  be as in \ref{equation:basic1} and \ref{equation:basic2}, we have the following results:
\begin{prop}\begin{itemize}
\item[(1)] For $ d \in K(S, 2), $ if one of the following conditions holds: \\
(a) \ $ p \mid d; $ \quad  (b) \ $ q \mid d; $ \quad (c) \ $ d = -1.
$ \ Then $ d \notin S^{(\varphi )} (E / K ). $
\item[(2)](a)  $ 2 \in S^{(\varphi )} (E / K ) \ \Longleftrightarrow \
p \equiv 3 (\bmod 4 ); $ \quad (b) \ $ - 2 \in S^{(\varphi )}
(E / K ) \
 \Longleftrightarrow \ p \equiv 1 (\bmod 4 ). $
\item[(1$^{\prime}$)] For $ d \in K(S, 2), $ if  $ 2 \mid d, $ then $ d \notin
S^{(\widehat{\varphi } )} (E^{\prime } / K ). $
\item[(2$^{\prime}$)] $ -1, \ p, \ q \in S^{(\widehat{\varphi } )} (E^{\prime } /
K ). $\end{itemize}\end{prop}
\begin{proof} We only consider the case $ \varepsilon = 1, $ the
others are similar. \begin{itemize}
\item[(1)](a) and (b)  follow directly by valuation.
 \item[(1)](c) To prove $ -1 \notin
S^{(\widehat{\varphi } )} (E^{\prime } / K ), $ we only need to prove $ C_{-1}(K_{2} ) = \emptyset . $
If not, then there exists $(z_{0}, w_{0}) \in C_{-1}(K_{2} ), $
such that \begin{equation} - w_{0}^{2} = 1 + 2 (p + q) z_{0}^{2} + 4 z_{0}^{4}
\end{equation}
If $ v_{2}( z_{0} ) \geq 0, $ then $ w_{0} = a + 2 b , $ where $ a \in \{ 1, \pi , -
\overline{\pi } \} $ and $ b \in \mathbb{O}_{K_{2}}, $ substituting
them into $ (4), $ we get $ v_{2}( a^{2} + 1 ) \geq 2,
$ which is impossible;
if $ v_{2}( z_{0} ) < 0, $  then $ v_{2}( w_{0} ) = 1 + 2 v_{2}( z_{0} ). $
Let $ z_{0} = 2^{-t} z_{1},
w_{0} = 2^{1 - 2 t} w_{1} $ with $ z_{1}, w_{1} \in
\mathbb{O}_{K_{2}}^{*},t \in \mathbb{Z}_{\geq 1}, $ substituting them into $ (4), $ we get $
v_{2}(w_{1}^{2} + z_{1}^{4} ) \geq 2, $ which is impossible. Therefore $
C_{-1}(K_{2} ) = \emptyset . $
\item[(2)](a) Let $ f(z, w) = - w^{2} + 2 - 2(p + q ) z^{2} + 2 z^{4}. $ \\
\indent i) \ To prove that $ C_{2}(K_{2})\neq \emptyset $ if and only if $p \equiv 3(\bmod 4). $ For necessity, first note that $ C_{2}(K_{2})\neq \emptyset $ implies $ C_{2}(K_{2})\subseteq \mathbb{O}_{K_{2}}^{\ast}\times 2\mathbb{O}_{K_{2}}.$ Taking any $(1 + 2 z_{0},2 w_{0} )\in C_{2}(K_{2})$ for some $ z_{0}, w_{0} \in
\mathbb{O}_{K_{2}},$ then they satisfy $$ w_{0}^{2}
= 8 z_{0}^{2}(1 + z_{0} )^{2} - p - 4 p z_{0}(1 + z_{0} ), $$ taking the valuation $v_{2}$ of both sides, we get $ p \equiv 3(\bmod 4 ). $
Conversely, let $ g(z, w) = 8 z^{2}(1 + z )^{2} - p - 4 p
z( 1 + z ) - w^{2}, $ by the above discussion, we know that $
C_{2} (K_{2}) \neq \emptyset \Longleftrightarrow \ g(z, w) = 0 $
has solutions in $ \mathbb{O}_{K_{2}}^{2}. $ \ If $ p \equiv 3 (\bmod 8 ),
$ then $ v_{2}(g(0, 2 + \sqrt{D})) > 2 v_{2}(g^{\prime }_{w}(0, 2
+ \sqrt{D} )); $ \ If $ p \equiv 7 (\bmod \ 8 ), $ then $
v_{2}(g(0, 1)) > 2 v_{2}(g^{\prime}_{w}(0, 1 )), $ by Hensel
lemma, $ C_{2}(K_{2} ) \neq \emptyset . $  \\
\indent ii) To prove that $C_{2}(K_{p}) \neq \emptyset. $ \ If $ p \equiv 1, 7 (\bmod \ 8 ), $ then $ w_{0}^{2}
\equiv 2 (\bmod \ p )$ for some $w_{0} \in  \mathbb{Z}$(note that $ (\frac{2} {p} )
= 1),$  we have $ v_{p}(f(0, w_{0} )) > 2
v_{p}(f^{\prime }_{w}(0, w_{0} )); $ If $ p \equiv 3, 5 (\bmod
\ 8 ), $ then $ Db^{2} \equiv 2 (\bmod p) $ for some $b \in \mathbb{Z}$( note that $ (\frac{2} {p}) = (\frac{D} {p} ) = - 1$ ), we have $ v_{p}(f(0, \sqrt{D} b)) > 2 v_{p}(f_{w}(0, \sqrt{D}
b )), $  by  Hensel lemma, $ C_{2}(K_{p}) \neq \emptyset. $ \\
\indent iii) To prove that $C_{2}(K_{q}) \neq \emptyset. $ Its proof is similar to ii). \\
Obviously, $C_{2}(K_{\infty})=C_{2}(\mathbb{C}) \neq \emptyset. $
To sum up, $ 2 \in S^{(\varphi )} (E / K ) \ \Longleftrightarrow \
p \equiv 3 (\bmod 4 ). $
\item[(2)](b) similar to (a).
\item[(1$^{\prime }$)]  follows directly from the property of
valuation.
\item[(2$^{\prime }$)] is similar to Proposition \ref{prop:mainA12}
(2)(3).\end{itemize}\end{proof}

\subsection{Case C \ $ K = \mathbb{Q} (\sqrt{-2}) $ }

We denote the condition: "$ K = \mathbb{Q} (\sqrt{-2}) $ both $ p $ and $ q $ are inertia in $
K $ " by condition (C). In this section, we always assume that condition (C) holds.\\
\indent By ramification theory, condition (C) holds if and only if
  $ \left( \frac{- 2}{p}\right) = \left( \frac{-
2}{q}\right) = - 1, $ which is also equivalent to that \ $ p \equiv 5
(\bmod \ 8 ). $ Note that $ 2 $ is totally ramifid in $ K, $
denote $ \pi _{2} =  \sqrt{ - 2}.$  Under the above assumption and notation, here $ S = \{ \infty , \pi
_{2} , p, q \}$ and $ K( S, 2) = <-1, \pi _{2}, p ,  q > . $ The
completions $ K_{v} $ of $ K $
at $ v \in S $ are given respectively by
$$ K_{\infty } = \mathbb{C}, \ K_{\pi _{2} } \cong \mathbb{Q}_{2 } (\sqrt{-2}), \
K_{p} = \mathbb{Q}_{p} (\sqrt{-2}), \ K_{q} = \mathbb{Q}_{q} (\sqrt{-2}). $$
Note that $ 4 = \pi _{2}^{4}, $ for each $ d \in K( S, 2), $ the
corresponding homogenous space can be transformed by variable
transformation $ z \mapsto \frac{z}{\pi _{2}} $ to the following form $$
 C_{d} :  \ d w^{2} = d^{2} + \varepsilon ( p + q) d z^{2} + z^{4}.
$$

\indent For  $ E
= E_{\varepsilon },E^{\prime }=E^{\prime }_{\varepsilon } $ with $ \varepsilon = \pm 1 $  be as in \ref{equation:basic1} and \ref{equation:basic2}, we have the following results:
\begin{prop}\begin{itemize}
\item[(1)] For $ d \in K(S, 2), $ if one of the following conditions holds: \\
(a) \ $ p \mid d; $ \quad  (b) \ $ q \mid d; $ \quad (c) \ $ \pi
_{2} \mid d. $  \quad Then $ d \notin S^{(\varphi )} (E / K ). $
\item[(2)] If $ \varepsilon = 1, $ then $ - 1 \notin S^{(\varphi )} (E
/ K ); $ \ If $ \varepsilon = -1, $ then $ - 1 \in S^{(\varphi )}
(E / K ). $
\item[(1$^{\prime }$)]  For $ d \in K(S, 2), $ if  $ \pi _{2} \mid d, $ Then $ d
\notin S^{(\widehat{\varphi } )} (E^{\prime } / K ). $
\item[(2$^{\prime }$)] $ - 1, p, \ q \in S^{(\widehat{\varphi } )} (E^{\prime } / K
). $ \end{itemize}\end{prop}

\begin{proof}\begin{itemize}
\item[(1)]  follows directly by
the valuation property.
\item[(2)]($ \varepsilon = 1 $) Let $
f(z, w ) = w^{2} + 1 - (p + q ) z^{2} + z^{4}. $\\
To prove that $ - 1 \notin S^{(\varphi )} (E
/ K ),$ we only need to prove that $ C_{-1}(K_{\pi _{2}}) = \emptyset. $ If not, note that  $ C_{-1}(K_{\pi _{2}}) \neq \emptyset $ implies $C_{-1}(\mathbb{O}_{K_{\pi _{2}}}) \neq \emptyset. $ But taking any $ (z,w )
\in \mathbb{O}_{K_{\pi _{2}}}^{2}, $ by explicit calculation, we get
$$ v_{\pi _{2}}( f(z, w) ) = \left \{
      \begin{array}{rl} 0, & \text{if} \ z, w
\in \mathbb{O}_{K_{\pi _{2}}}^{*} \text{or} \  z, w \in
\pi _{2} \mathbb{O}_{K_{\pi _{2}}},  \\
2,  & \text{if} \ v_{\pi _{2}}(w - 1 ) \geq 2 \
\text{and} \ z \in \pi _{2} \mathbb{O}_{K_{\pi _{2}}} \ \text{or} \
v_{\pi _{2}}( w ) \geq 2 \ \text{and} \ z \in \mathbb{O}_{K_{\pi _{2}}} ^{*},  \\
3, & \text{if} \ v_{\pi _{2}}(w - 1) = 1 \ \text{and} \
z \in \pi _{2} \mathbb{O}_{K_{\pi _{2}}}, \\
4, & \text{if} \ v_{\pi _{2}}(w - \pi _{2}) \geq 3 \
\text{and} \ z \in \mathbb{O}_{K_{\pi _{2}}}^{*}, \\
5, & \text{if} \ v_{\pi _{2}}(w - \pi _{2} ) = 2 \
\text{and} \ v_{\pi _{2}}( z - 1) \geq 2, \\
6, & \text{if} \ v_{\pi _{2}}(w - \pi _{2}) = 2 \
\text{and} \ v_{\pi _{2}}(z - 1) = 1.
\end{array}
    \right.$$

which implies $ f(z, w )=0 $ has no solutions in $
\mathbb{O}_{K_{\pi_{2}}}^{2},$ this is a contradict.
\item[(2)]($ \varepsilon = -1. $) Let $ f(z, w) =
w^{2} + 1 + (p + q) z^{2} + z^{4}. $ \\
\indent i) To prove that $
C_{-1}(K_{p} )
\neq \emptyset. $  Since $ p \equiv 5 (\bmod \ 8 ),$ then $ a^{2} \equiv -1 (\bmod \ p )$ for some $ a \in
\mathbb{Z}. $ By $
v_{p}(f(0, a ))
> 2 v_{p}(f^{\prime }_{w}(0, a ))$ and Hensel lemma, $
C_{-1}(K_{p} )
\neq \emptyset. $ \\
\indent ii)To prove that $
C_{-1}(K_{q} )
\neq \emptyset. $  Since $ q \equiv 7 (\bmod \ 8 ),$ then $ 2 b^{2} \equiv 1 (\bmod \
q ) $ for some $ b
 \in \mathbb{Z}. $ By $ v_{q}(f(0, \sqrt{-2} b )) >
 2 v_{q}(f_{w}(0, \sqrt{-2} b )) $ and Hensel lemma, $ C_{-1}(K_{q})
\neq \emptyset . $ \\
\indent iii) To prove that $
C_{-1}(K_{\pi _{2}} )
\neq \emptyset. $ By $ v_{\pi _{2}}(f(1
 + \pi _{2}, \pi _{2} + \pi _{2}^{2} )) > 2 v_{\pi _{2}}(f^{\prime}_{w}
(1 + \pi _{2}, \pi _{2} + \pi _{2}^{2} )) $ and  Hensel lemma, $
C_{-1}(K_{\pi _{2}}) \neq \emptyset . $ \\
Obviously, $C_{-1}(K_{\infty})=C_{2}(\mathbb{C}) \neq \emptyset. $
To sum up, $ -1 \in
S^{(\varphi )}(E / K ). $\\
\indent (1$^{\prime }$)  follows directly by
the valuation property.
\item[(2$^{\prime }$)] is similar to (2). \end{itemize}\end{proof}
\par  \vskip 0.2 cm

\subsection{ Case D \ $ K = \mathbb{Q} (\sqrt{-1}) $ }

We denote the condition: "$ K = \mathbb{Q} (\sqrt{-1}) $" by condition (D).
In this condition, $ 2 $ totally ramifies in $ K, $ denote $ \pi _{2}
= 1 - i, $ where $ i = \sqrt{-1}. $ Note that $ p $
and $ q $ can't be simultaneously inertia in $ K $ ( $ q - 2 =
p ).$ So we discuss the following two cases according to $p (\bmod)$:

\subsubsection{ Case D$_{1} $ }\label{case:D1}  Assume that $ p \equiv 1 (\bmod \ 4 ), $
\ then $ p $ splits completely in $ K. $ Denote $ p = \mu \cdot
\bar{\mu } $, where $ \mu , \bar{\mu } \in \mathbb{Z} [\sqrt{-1}] $ are two
conjugate irreducible elements. Obviously, $ q $ is inertia in $ K. $
In this case, \ $ S = \{ \infty , \pi _{2} , \mu , \bar{\mu }, q \}
$ and $ K( S, 2) = <i,  \pi _{2}, \mu , \bar{\mu }, q >. $ The
completions $ K_{v} $ at
$ v \in S $ are given respectively by $$
 K_{\infty } = \mathbb{C}, \ K_{\pi _{2} } \cong \mathbb{Q}_{2 } (\sqrt{-1}), \
K_{\mu } \cong K_{\bar{\mu }} \cong \mathbb{Q}_{p}, \ K_{q} = \mathbb{Q}_{q}
(\sqrt{-1}). $$ Note that $ 2 = i \cdot \pi _{2}^{2}, \ 4 = -\pi
_{2}^{4}, $ by variable transformations $ z \mapsto \pi _{2} z $ and
$ z \mapsto i z, $ for any $ d \in K( S, 2), $ the corresponding
homogenous spaces can be given
respectively by
$$ C_{d}: \ d w^{2} = d^{2} - \varepsilon ( p + q ) i d z^{2} -
z^{4}, $$
$$ C_{d}^{\prime } : \ d w^{2} = d^{2} - \varepsilon ( p + q ) d
z^{2} + p q z^{4}. $$
\par  \vskip 0.1 cm
\indent For  $ E
= E_{\varepsilon },E^{\prime }=E^{\prime }_{\varepsilon } $ with $ \varepsilon = \pm 1 $  be as in \ref{equation:basic1} and \ref{equation:basic2}, we have the following results:
\begin{prop}\begin{itemize}
\item[(1)]  For $ d \in K(S, 2), $ if one of the following conditions holds: \\
(a) \ $ \mu \mid d; $ \quad  (b) \ $ \bar{\mu } \mid d; $ \quad  (c)
\ $ q \mid d; $ \quad (d) \ $ \pi _{2} \mid d. $ \ Then $ d \notin
S^{(\varphi )} (E / K ). $
\item[(2)] $ i \in S^{(\varphi )} (E / K ) \
 \Longleftrightarrow \ p \equiv 1 (\bmod \ 8 ). $ \end{itemize}\end{prop}
\begin{prop} \begin{itemize}
\item[(1)] \ For $ d \in K(S, 2), $ if one of the following conditions
holds: \\
(a) \ $ \pi _{2} \mid d; $ \quad  (b) \ $ d = i. $ \quad Then $ d
\notin S^{(\widehat{\varphi } )} (E^{\prime } / K ). $
\item[(2)] $ p, \ q \in S^{(\widehat{\varphi } )} (E^{\prime } / K ). $
\item[(3)] (a) \ $ \mu \in S^{(\widehat{\varphi } )} (E^{\prime } / K ) \
\Longleftrightarrow $ the imaginary part
$ \Im \mu \equiv 0 (\bmod \ 4 ); $ \\
(b) \ $ \bar{\mu } \in S^{(\widehat{\varphi } )} (E^{\prime } / K )
\ \Longleftrightarrow $ the imaginary part
$ \Im \bar{\mu } \equiv 0 (\bmod \ 4 ); $ \\
(c) \ $ i \mu \in S^{(\widehat{\varphi } )} (E^{\prime } / K ) \
\Longleftrightarrow $ the real part $ \Re \mu \equiv 0 (\bmod \ 4 ); $ \\
(d) \ $ i \bar{\mu } \in S^{(\widehat{\varphi } )} (E^{\prime } / K
) \ \Longleftrightarrow $ the real part of $ \Re \bar{\mu } \equiv 0
(\bmod \ 4 ). $ \end{itemize}\end{prop}
\par  \vskip 0.1 cm
\begin{proof}  We only prove (1) (b)
and (3) (a), the others are similar.
\begin{itemize} \item[(1)](b) \ Let $ f(z, w ) =  - i w^{2} - 1 - (p +
q) i z^{2} + p q z^{4}. $
Wantting to  prove $ i
\notin S^{(\widehat{\varphi } )} (E^{\prime } / K ),$ we only need to prove that $ C_{i}^{'}(K_{\pi _{2}}) = \emptyset. $ Taking any $(z, w ) \in K_{\pi _{2}}^{2},
$
by computation, we get \begin{equation}v_{\pi _{2}}(f(z,w))= \left\{
      \begin{array}{rl}
2, \hspace{2mm} & \text{ if } v_{\pi _{2}}(w) < 0; \\
 1, \hspace{2mm} & \text{ if } v_{\pi _{2}}(w) = 0; \\
3, \hspace{2mm} & \text{ if } v_{\pi _{2}}(w) > 0,
\end{array}
\right.
\end{equation}
which implies that $f(z,w)=0$ has no solutions in $K_{\pi _{2}}^{2},$
 therefore $ C_{i}^{'}(K_{\pi _{2}}) =
\emptyset.$
\item[(3)](a) Let $ f(z, w ) = \mu ^{2} - (p + q) \mu z^{2} + p q z^{4}
- \mu w^{2}. $ \ Since $ p \equiv 1 (\bmod \ 4 ), $ then $ p =
a^{2} + b^{2} $ and $ \mu = a + b i = a + b - b \pi_{2} $ for some
$ a, b \in \mathbb{Z}. $   \\
\indent i) To prove that $ C_{\mu}^{'}(K_{\pi _{2}}) \neq \emptyset$ if and only if $ b\equiv 0 (\bmod 4).$ For necessity, it's equivalent to prove that if $b\equiv 1,2,3 (\bmod 4),$ then $ C_{\mu}^{'}(K_{\pi _{2}}) = \emptyset.$ Let $ (z, w
) \in \mathbb{O}_{K_{\pi _{2}}}^{2}, $ by explicit calculation, we
get
\begin{equation}\label{equation:D11}v_{\pi _{2}}(f(z, w )) \leq \left\{
      \begin{array}{rl}
0, \hspace{2mm} & \text{ if } \  z, w \in
 \mathbb{O}_{K_{\pi _{2}}}^{*} \  \text{ or } \ z, w \in
 \pi _{2} \mathbb{O}_{K_{\pi _{2}}}, \\
 3, \hspace{2mm} & \text{ if } \ z \in \pi_{2}\mathbb{O}_{K_{\pi_{2}}},
 w \in \mathbb{O}_{K_{\pi_{2}}}^{*}, \\
5, \hspace{2mm} & \text{ if } z \in \mathbb{O}_{K_{\pi_{2}}}^{*}, w
\in \pi _{2} \mathbb{O}_{K_{\pi _{2}}}.
\end{array}
\right.
\end{equation}
which implies that $f(z,w)=0 $ has no solutions in $\mathbb{O}_{K_{\pi _{2}}}^{2}.$ Putting $ g(z_{1},
w_{1}) = \mu z_{1}^{4} - (p + q) \mu z_{1}^{2} + p q - \mu
w_{1}^{2}, $ then $g(z_{1},w_{1})=z_{1}^{4}f(\frac{1}{z_{1}},\frac{w_{1}}{z_{1}^{2}}).$ By the above discussion, we know that $ f(z, w ) = 0 $ has
solutions in $ K_{\pi _{2}}^{2} $ if and only if  $ g(z_{1}, w_{1})=0
$ has solutions in $ \mathbb{O}_{K_{\pi _{2}}}^{2}. $ Similarly, one
can get
\begin{equation}\label{equation:D12}v_{\pi _{2}}(g(z_{1}, w_{1})) \leq \left\{
\begin{array}{rl}
0, \hspace{2mm} & \text{ if } \ z_{1}, w_{1} \in \mathbb{O}_{K_{\pi
_{2}}}^{*} \ \text{or} \  z_{1}, w_{1} \in \pi _{2} \mathbb{O}_{K_{\pi_{2}}}, \\
3, \hspace{2mm} & \text{ if } \ z_{1} \in \pi _{2} \mathbb{O}_{K_{\pi_{2}}},
w_{1} \in \mathbb{O}_{K_{\pi_{2}}}^{*}, \\
5, \hspace{2mm} & \text{ if } \ z_{1} \in \mathbb{O}_{K_{\pi
_{2}}}^{*}, w_{1} \in \pi _{2} \mathbb{O}_{K_{\pi_{2}}}.
\end{array}
    \right.
\end{equation}
which implies that $g(z_{1}, w_{1})=0$ has no solutions in $\mathbb{O}_{K_{\pi _{2}}}^{2}.$ By \ref{equation:D11} and \ref{equation:D12}, we conclude that $ C_{\mu}^{'}(K_{\pi _{2}}) = \emptyset.$
Therefore if $ C_{\mu}^{'}(K_{\pi _{2}}) \neq \emptyset, $ then the
imaginary part $ b $ of $ \mu $ must satisfy
$ b \equiv 0 (\bmod \ 4 ). $
Conversely, if $ b \equiv 0 (\bmod \ 4 ), $ then we have the
following results: \\
If $ (b, a) \equiv (4, 3 ), ( 0, 7)(\bmod \ 8), $ then $ v_{\pi
_{2}}(g(\pi _{2}, 1))
> 2 v_{\pi _{2}}(g^{\prime }_{w_{1}}(\pi _{2}, 1)); $ \\
If $ (b, a) \equiv (4, 7), ( 0, 3) (\bmod \ 8), $ then $ v_{\pi
_{2}}(g(0, 1)) > 2 v_{\pi _{2}}(g^{\prime }_{w_{1}}(0, 1)); $ \\
If $ (b, a) \equiv (4, 5 ), ( 0, 1)  (\bmod \ 8), $ then $
v_{\pi _{2}}(g( \pi _{2}, 1 + \pi _{2})) > 2 v_{\pi _{2}}
(g^{\prime }_{w_{1}}(\pi _{2}, 1 + \pi _{2})); $ \\
If $ (b, a) \equiv (4, 1 ), ( 0, 5)(\bmod \ 8), $ then $ v_{\pi
_{2}}(g( 0, 1 + \pi_{2})) > 2 v_{\pi _{2}}(g^{\prime }_{w_{1}}
(0, 1 + \pi_{2})), $ by the above results and Hensel lemma, we get $ C_{\mu}^{'}(K_{\pi_{2}}) \neq \emptyset.  $ \\
\indent ii) By lemma 14 in \cite{Merriman}, we can easily obtain $ C_{\mu }^{'}(K_{\mu })
\neq \emptyset, C_{\mu }^{'}(K_{\bar{\mu }}) \neq \emptyset, C_{\mu
}^{'}(K_{q}) \neq \emptyset. $
To sum up, we prove (3)(a). \end{itemize}\end{proof}

\subsubsection{Case D$_{2} $}\label{case:D2} Assume that $ p \equiv 3 (\bmod \ 4 ), $
\ then $ p $ is inertia in $ K, $ while $ q $  splits completely in
$ K. $ Denote $ q = \mu \cdot \bar{\mu }, $, where $ \mu , \bar{\mu } \in \mathbb{Z} [\sqrt{-1}] $ are two
conjugate irreducible elements.
In this case, \ $ S = \{
\infty , \pi _{2} , \mu , \bar{\mu }, p \} $ and $ K( S, 2) = <i,
 \pi _{2},  \mu , \bar{\mu }, p >. $ The completions $ K_{v} $ at
 $ v \in S $ are given respectively by $$
K_{\infty } = \mathbb{C}, \ K_{\pi _{2} } \cong \mathbb{Q}_{2 } (\sqrt{-1}), \
K_{\mu } \cong K_{\bar{\mu }} \cong \mathbb{Q}_{q}, \ K_{p} = \mathbb{Q}_{p}
(\sqrt{-1}). $$ Similarly as the above case D$_{1}, $ by variable
transformations, the corresponding homogenous spaces can be given
respectively by
$$ C_{d}: \ d w^{2} = d^{2} - \varepsilon ( p + q ) i d z^{2} -
z^{4}, $$
$$ C_{d}^{\prime } : \ d w^{2} = d^{2} - \varepsilon ( p + q ) d
z^{2} + p q z^{4}. $$
\par  \vskip 0.1 cm
\indent For  $ E
= E_{\varepsilon },E^{\prime }=E^{\prime }_{\varepsilon } $ with $ \varepsilon = \pm 1 $  be as in \ref{equation:basic1} and \ref{equation:basic2}, we have the following results:
\begin{prop}\begin{itemize}
\item[(1)] For $ d \in K(S, 2), $ if one of the following conditions holds: \\
(a) \ $ \mu  \mid d; $ \quad  (b) \ $ \bar{\mu } \mid d; $ \quad
(c) \ $ p \mid d; $ \quad (d) \ $ \pi _{2} \mid d. $ \ Then $ d
\notin S^{(\varphi )} (E / K ). $\\
\item[(2)] $ i \in S^{(\varphi )} (E / K ) \
 \Longleftrightarrow \ p \equiv 7 (\bmod \ 8 ). $\end{itemize}\end{prop}
\begin{proof} we only prove (2).\\
(2) Let $ f(z, w ) = iw^{2} + 1 - (p + q)z^{2} +
z^{4}. $ \\
\indent i)  To prove that $ C_{i}(K_{\pi_{2}}) \neq \emptyset $ if and only if $p\equiv 7(\bmod 8).$ For necessity, we prove that if $ p \equiv 3 (\bmod \ 8 ), $ then $ C_{i}(K_{\pi_{2}}) = \emptyset. $ Note that $ C_{i}(K_{\pi_{2}}) \neq \emptyset $ implies
$ C_{i}(\mathbb{O}_{K_{\pi_{2}}}) \neq
\emptyset. $ Let $ (z, w ) \in
\mathbb{O}_{K_{\pi _{2}}}^{2}, $ by the explicit calculation, we get
\begin{equation}v_{\pi _{2}}(f(z, w)) \leq \left\{
      \begin{array}{rl}
0, \hspace{2mm} & \text{ if } \ z, w \in \mathbb{O}_{K_{\pi
_{2}}}^{\ast} \
\text{or} \ z, w \in \pi _{2}\mathbb{O}_{K_{\pi _{2}}}, \\
1, \hspace{2mm} & \text{ if } \ z
 \in \pi _{2}\mathbb{O}_{K_{\pi _{2}}}, w \in
\mathbb{O}_{K_{\pi _{2}}}^{\ast}, \\
2, \hspace{2mm} & \text{ if } \ v_{\pi _{2}}(w) > 1 , z \in
\mathbb{O}_{K_{\pi _{2}}}^{\ast}, \\
4,\hspace{2mm} & \text{ if } \  v_{\pi _{2}}(w  - \pi _{2})  \geq  3
, z \in \mathbb{O}_{K_{\pi _{2}}}^{\ast}, \\
6,\hspace{2mm} & \text{ if } \  v_{\pi _{2}}(w  - \pi _{2}  - \pi
_{2}^{2}) \geq 3 , z \in \mathbb{O}_{K_{\pi _{2}}}^{\ast}.
\end{array}
    \right.
\end{equation}
which implies that $ f(z, w )=0 $ has no solutions in $
\mathbb{O}_{K_{\pi _{2}}}^{2} $ .
 Therefore we have proved that if $ C_{i}(K_{\pi_{2}})
\neq \emptyset, $ then $ p \equiv 7 (\bmod \ 8). $
Conversely, if $ p \equiv 7 (\bmod \ 8), $ by $ v_{\pi
_{2}}(f(1, \pi _{2} + \pi _{2}^{2}))
> 2 v_{\pi _{2}}(f_{w}(1, \pi _{2} + \pi _{2}^{2}))$ and Hensel
lemma, $ C_{i}(K_{\pi _{2}}) \neq \emptyset. $ \\
\indent ii) To prove that if $p\equiv 7 (\bmod 8),$ then $ C_{i}(K_{\mu})\neq \emptyset. $ \ If $ p \equiv  7 (\bmod \ 8 ), $ then  there exists $ a \in \mathbb{Z} $ such that $ 2 a^{2} \equiv  1
(\bmod \ p). $ By $ v_{p}(f(0, a + i a ))
 > 2 v_{p}(f^{\prime }_{w}(0, a + i a )) $ and Hensel lemma, we get
$ C_{i}(K_{\mu})\neq \emptyset . $  \\
\indent iii) To prove that if $p\equiv 7 (\bmod 8),$ then $ C_{i}(K_{\bar{\mu }})\neq \emptyset. $ Its proof is similarly to ii). \\
\indent iv) To prove that if $p\equiv 7 (\bmod 8),$ then $ C_{i}(K_{q} ) \neq \emptyset. $ If $ q \equiv 1 (\bmod \ 8 ), $ then there exists $a \in \mathbb{Z} $ such that $  2 a^{2}
\equiv
 1 (\bmod \ q). $ By $ v_{p}(f(0, a + i a )) > 2
v_{p}(f^{\prime }_{w}(0, a + i a )) $ and  Hensel lemma
$ C_{i}(K_{q} ) \neq \emptyset. $ \\
Therefore we have proved $ i \in S^{(\varphi)}(E / K) \
\Longleftrightarrow \ p \equiv 7 (\bmod \ 8). $ \end{proof}
\begin{prop}
\begin{itemize}
\item[(1)]  For $ d \in K(S, 2), $ if $ \pi _{2} \mid d, $ then $ d \notin
S^{(\widehat{\varphi } )} (E^{\prime } / K ). $
\item[(2)] $ p, \ q \in S^{(\widehat{\varphi } )} (E^{\prime } / K ). $
 \item[(3)] \ $ i \in S^{(\widehat{\varphi } )} (E^{\prime } / K )
\Longleftrightarrow \ p \equiv 7 (\bmod \ 8 ). $
\item[(4)] (a) $ \mu \in S^{(\widehat{\varphi } )} (E^{\prime } / K ) \
\Longleftrightarrow \ p \equiv 7 (\bmod \ 8 ) $ \ or \
the real part $ \Re \mu  \equiv 2 (\bmod 4 ); $ \\
(b)  $ \bar{\mu } \in S^{(\widehat{\varphi } )} (E^{\prime } / K )
\ \Longleftrightarrow \ p \equiv 7 (\bmod \ 8 ) $ \ or \
the real part $ \bar{\mu } \equiv 2 (\bmod \ 4 ); $ \\
(c) $ i \mu \in S^{(\widehat{\varphi } )} (E^{\prime } / K ) \
\Longleftrightarrow \ p \equiv 7 (\bmod \ 8 ) $ \ or \
the imaginary part $ \Im \mu \equiv 2 (\bmod \ 4 ); $ \\
(d) $ i \bar{\mu } \in S^{(\widehat{\varphi } )} (E^{\prime } / K )
\ \Longleftrightarrow \ p \equiv 7 (\bmod \ 8 ) $ \ or \ the
imaginary part  $ \Im \bar{\mu } \equiv 2 (\bmod \ 4 ). $\end{itemize}\end{prop}
\begin{proof} We only prove (3) and (4)(a).
\begin{itemize}
\item[(3)]  \ Let $ f(z, w ) =  - i w^{2} - 1 - (p + q) i
z^{2} + p q z^{4}. $ \\
 \indent i)  To prove that $ C_{i}^{'}(K_{\pi_{2}}) \neq \emptyset $ if and only if $ p\equiv 7(\bmod 8).$ For necessity, we prove that if $ p \equiv 3 (\bmod \ 8 ), $ then $ C_{i}^{'}(K_{\pi_{2}}) = \emptyset. $
 Let $ (z, w ) \in
\mathbb{O}_{K_{\pi _{2}}}^{2}, $ by explicitly calculating, we get
 \begin{equation}v_{\pi _{2}}(f(z, w
)) \leq \left\{
      \begin{array}{rl}
0, \hspace{2mm} & \text{ if } \ z, w \in \mathbb{O}_{K_{\pi
_{2}}}^{\ast} \ \text{or} \ z, w \in \pi _{2}\mathbb{O}_{K_{\pi _{2}}}, \\
1, \hspace{2mm} & \text{ if } \ z
 \in \pi _{2}\mathbb{O}_{K_{\pi _{2}}}, w \in
\mathbb{O}_{K_{\pi _{2}}}^{\ast }, \\
2, \hspace{2mm} & \text{ if } \ v_{\pi _{2}}(w) > 1, \ z \in
\mathbb{O}_{K_{\pi _{2}}} ^{\ast}, \\
4, \hspace{2mm} & \text{ if } \ v_{\pi _{2}}(w - \pi _{2}) > 2, \
 z \in \mathbb{O}_{K_{\pi _{2}}} ^{\ast}, \\
6, \hspace{2mm} & \text{ if } \ v_{\pi_{2}}(w - \pi _{2} - \pi
_{2}^{2})  >  2 , \ z \in \mathbb{O}_{K_{\pi _{2}}}^{\ast }.
\end{array}
    \right.
\end{equation}
which implies that $ f(z, w )=0 $ has no solutions in $
\mathbb{O}_{K_{\pi _{2}}}^{2}. $  Thus if
$ ( z, w ) \in C_{i}^{'}(K_{\pi_{2}}), $ then $ z = \pi_{2}^{ - r} z_{0}, w = \pi_{2}^{ - 2r}
w_{0} $ with $ r \geq 1, z_{0}, w_{0} \in O_{K_{\pi _{2}}}^{*}. $
Substituting them into $ f(z, w ), $ we get $ v_{\pi _{2}}(f(z,w))
= 2, $ a contradiction. Therefore we have proved that if $
C_{i}^{'}(K_{\pi_{2}}) \neq \emptyset, $ then $ p \equiv 7 (\bmod \ 8). $
Conversely, if $ p \equiv  7 (\bmod \ 8 ), $ by $ v_{\pi
_{2}}(f(1 + \pi _{2}, \pi _{2} + \pi _{2}^{2} )) > 2 v_{\pi
_{2}}(f^{\prime }_{w}(1 + \pi _{2}, \pi _{2} + \pi _{2}^{2} )) $ and Hensel lemma, $ C_{i}^{'}(K_{\pi _{2}}) \neq \emptyset. $ \\
\indent ii) \ By lemma 14 in \cite{Merriman}, $ C_{i}^{'}(K_{\mu } ) \neq \emptyset,
 C_{i}^{'}(K_{\mu }) \neq \emptyset,  C_{i}^{'}(K_{\bar{\mu }}) \neq \emptyset,
 C_{i}^{'}(K_{q}) \neq \emptyset . $ Therefore we have proved $ i \in S^{(\varphi)}(E^{\prime } / K ) \
\Longleftrightarrow \ p \equiv 7 (\bmod \ 8). $
\item[(4)(a)]  Let $ f(z, w ) = \mu ^{2} - (p + q) \mu z^{2} + p q
z^{4} - \mu w^{2}. $
Since $ q \equiv 1 (\bmod \ 4), $ then $ a^{2} + b^{2} = q, a,
b \in \mathbb{Z}, \mu = a + b i = a + b - b \pi _{2}. $ \\
\indent i) To prove that $ C_{\mu}^{'}(K_{\pi _{2}}) \neq \emptyset$ if and only if $ a\equiv 2 (\bmod 4)$ or $p\equiv 7 (\bmod 8).$ For necessity, it's equivalent to prove that if $b\equiv 2 (\bmod 4),$ then $ C_{\mu}^{'}(K_{\pi _{2}}) = \emptyset.$ Let $ (z, w
) \in \mathbb{O}_{K_{\pi _{2}}}^{2}, $ by explicit calculation, we
get
\begin{equation}\label{equation:D1}v_{\pi _{2}}(f(z, w)) \leq \left\{
      \begin{array}{rl}
0, \hspace{2mm} & \text{ if } \ z, w \in \mathbb{O}_{K_{\pi
_{2}}}^{\ast } \ \text{or} \  z, w \in \pi _{2}\mathbb{O}_{K_{\pi _{2}}}, \\
3, \hspace{2mm} & \text{ if } \ z \in \pi _{2} \mathbb{O}_{K_{\pi
_{2}}}, \  w \in \mathbb{O}_{K_{\pi _{2}}}^{\ast }, \\
5, \hspace{2mm} & \text{ if } \ z \in \mathbb{O}_{K_{\pi
_{2}}}^{\ast }, \  w \in \pi _{2}\mathbb{O}_{K_{\pi _{2}}}
\end{array}
    \right.
\end{equation}
which implies that $f(z,w)=0 $ has no solutions in $\mathbb{O}_{K_{\pi _{2}}}^{2}.$ Putting $ g(z_{1},
w_{1}) = \mu z_{1}^{4} - (p + q) \mu z_{1}^{2} + p q - \mu
w_{1}^{2}, $ then $g(z_{1},w_{1})=z_{1}^{4}f(\frac{1}{z_{1}},\frac{w_{1}}{z_{1}^{2}}).$ By the above discussion, we know that $ f(z, w ) = 0 $ has
solutions in $ K_{\pi _{2}}^{2} $ if and only if  $ g(z_{1}, w_{1})=0
$ has solutions in $ \mathbb{O}_{K_{\pi _{2}}}^{2}. $ Similarly, one
can get
\begin{equation}\label{equation:D2} v_{\pi _{2}}(g(z_{1}, w_{1})) \leq \left\{
      \begin{array}{rl}
0,\hspace{2mm} & \text{ if } \  z_{1}, w_{1} \in \mathbb{O}_{K_{\pi
_{2}}}^{*}  \ \text{or} \  z_{1}, w_{1} \in
\mathbb{O}_{K_{\pi _{2}}}, \\
3,\hspace{2mm} & \text{ if } \ z_{1} \in
\pi_{2}\mathbb{O}_{K_{\pi_{2}}}, \
w_{1} \in \mathbb{O}_{K_{\pi _{2}}}^{*}, \\
2,\hspace{2mm} & \text{ if } \ v_{\pi _{2}}(w_{1}) = 1 , \
z \in \mathbb{O}_{K_{\pi _{2}}} ^{*},\\
4,\hspace{2mm} & \text{ if } \ v_{\pi_{2}}(w_{1}) = 2, \
z \in \mathbb{O}_{K_{\pi _{2}}} ^{*},\\
5,\hspace{2mm} & \text{ if } \ v_{\pi _{2}}(w_{1}) > 2, \ z \in
\mathbb{O}_{K_{\pi _{2}}} ^{*}.
\end{array}
    \right.
\end{equation}
which implies that $g(z_{1}, w_{1})=0$ has no solutions in $\mathbb{O}_{K_{\pi _{2}}}^{2}.$ By \ref{equation:D1} and \ref{equation:D2}, we conclude that $ C_{\mu}^{'}(K_{\pi _{2}}) = \emptyset.$
Therefore, if $ C_{\mu}^{'}(K_{\pi_{2}}) \neq \emptyset, $ then $ \ p
\equiv 7 (\bmod \ 8 ) $ or the real part $ \Re \mu \equiv 2 (\bmod 4 ). $
Conversely, we have the following results: \\
If $ (b, a) \equiv (4, 3 ), ( 0, 7)(\bmod \ 8), $ then $ v_{\pi
_{2}}(g(\pi _{2}, 1))
> 2 v_{\pi _{2}}(g^{\prime }_{w_{1}}(\pi _{2}, 1)); $ \\
If $ (b, a) \equiv (4, 7), ( 0, 3) (\bmod \ 8), $ then $ v_{\pi
_{2}}(g(0, 1)) > 2 v_{\pi _{2}}(g^{\prime }_{w_{1}}(0, 1)); $ \\
If $ (b, a) \equiv (4, 5 ), ( 0, 1)  (\bmod \ 8), $ then $
v_{\pi _{2}}(g( \pi _{2}, 1 + \pi _{2})) > 2 v_{\pi _{2}}
(g^{\prime }_{w_{1}}(\pi _{2}, 1 + \pi _{2})); $ \\
If $ (b, a) \equiv (4, 1 ), ( 0, 5)(\bmod \ 8), $ then $ v_{\pi
_{2}}(g( 0, 1 + \pi_{2})) > 2 v_{\pi _{2}}(g^{\prime }_{w_{1}}
(0, 1 + \pi_{2})). $ \\
If $ (b, a) \equiv (3, 0 ), ( 1, 2)(\bmod \ 4), $ then $ v_{\pi
_{2}}(g(1, \pi_{2})) > 2 v_{\pi _{2}}(g^{\prime }_{w_{1}}( 1, \pi_{2})). $ \\
If $ (b, a) \equiv (3, 2 ), ( 1, 0)(\bmod \ 4), $ then $ v_{\pi
_{2}}(g( 1, \pi_{2} + \pi_{2}^{2})) > 2 v_{\pi _{2}}(g^{\prime
}_{w_{1}}( 1, \pi_{2} + \pi_{2}^{2})), $ by the above results and Hensel lemma, $ C_{\mu}^{'}(K_{\pi_{2}}) \neq \emptyset .  $ \\
\indent ii) By lemma 14 in \cite{Merriman}, we can easily obtain $ C_{\mu }^{'}(K_{\mu })
\neq \emptyset, C_{\mu }^{'}(K_{\bar{\mu }}) \neq \emptyset, C_{\mu
}^{'}(K_{p}) \neq \emptyset. $ This proves (4)(a). \end{itemize}\end{proof}

\subsection{ Case E \ $ K = \mathbb{Q} (\sqrt{-3}) $ }

We denote the condition: "$ K = \mathbb{Q} (\sqrt{-3}) $ and $ p\equiv 2 (\bmod 3) $ " by condition (E). In this section, we always assume that condition (E) holds.\\
\indent By ramification theory, condition (E) holds if and only if
 $p$ is inertia in $K,$ while $q$ splits completely in $K.$ Denote $ q
= \mu \cdot \bar{\mu }, $ where $ \mu , \bar{\mu } \in \mathbb{Z} [\frac{1 +
\sqrt{-3}}{2} ] $ are two conjugate irreducible elements. Note that
$ 2 $ is inertia in $ K, $ the residual field $ \mathbb{O}_{K} / ( 2
\mathbb{O}_{K} ) \cong \mathbb{F}_{4}, $ the field of four elements.
Then we can take its representatives $ T = \{ 0, 1, \tau = \frac{ -1
+ \sqrt{-3} }{ 2 }, \tau ^{2} \}. $  Under the above assumption and notation, here $ S = \{ \infty , 2, \mu ,
\bar{\mu }, p \}, $ and \ $ K( S, 2) = <-1,  2, \mu , \bar{\mu }, p
>. $ The completions $ K_{v} $ at $ v \in S $ are given
respectively by
$$ K_{\infty } = \mathbb{C}, \ K_{2 } \cong \mathbb{Q}_{2} (\sqrt{-3}), K_{\mu }
\cong K_{\bar{\mu }} \cong \mathbb{Q}_{q}, \ K_{p} = \mathbb{Q}_{p} (\sqrt{-3}). $$
Note that we can assume that $ \mu = s + t \tau, $ where $ s, t \in
\mathbb{Z}, $ then $ \bar{\mu } = s - t - t \tau . $\\
\indent For  $ E
= E_{\varepsilon },E^{\prime }=E^{\prime }_{\varepsilon } $ with $ \varepsilon = \pm 1 $  be as in \ref{equation:basic1} and \ref{equation:basic2}, we have the following results:
\begin{prop}\begin{itemize}
\item[(1)] For $ d \in K(S, 2), $ if one of the following conditions holds: \\
(a) \ $ \mu \mid d; $ \quad  (b) \ $ \bar{\mu } \mid d; $ \quad  (c)
\ $ p \mid d; $ \quad (d) \ $ d = -1. $ \ Then $ d \notin
S^{(\varphi )} (E / K ). $
\item[(2)](a) $ 2 \in S^{(\varphi )}(E / K ) \Longleftrightarrow \
p \equiv 23 (\bmod \ 24 ); $ \\
(b) $ -2 \in S^{(\varphi)}(E / K ) \Longleftrightarrow \ p \equiv 17
(\bmod \ 24 ). $\end{itemize}\end{prop}
\begin{prop}\begin{itemize}
\item[(1)] For $ d \in K(S, 2), $ if $ 2 \mid d, $ Then $ d \notin
S^{(\widehat{\varphi } )} (E^{\prime } / K ). $
\item[(2)] $ -1, \ p, \ q \in S^{(\widehat{\varphi} )} (E^{\prime } / K
). $
\item[(3)](a) \ $ \mu \in S^{(\widehat{\varphi} )} (E^{\prime } / K ) \
\Longleftrightarrow \ p \equiv 17, 23 (\bmod \ 24 ); $ \\
(b) \ $ \bar{\mu } \in S^{(\widehat{\varphi} )} (E^{\prime } / K ) \
\Longleftrightarrow \ p \equiv 17, 23 (\bmod \ 24 ). $\end{itemize}\end{prop}
\begin{proof} We only prove (3)(a). \\
 (3) (a) \ Let $ f(z, w ) =
 - \mu w^{2} + \mu ^{2} + (p + q ) \mu z^{2} + p q z^{4}. $ \\
\indent i) To prove that $ C_{\mu }^{'}(K_{2}) \neq \emptyset $ if and only if $p\equiv 17,23 (\bmod 24).$ For necessity, if $ C_{\mu }^{'}(K_{2}) \neq \emptyset, $ taking any $(z,w)\in C_{\mu }^{'}(K_{2}). $
 If $ v_{2}(w)
= 0, \ v_{2}(z) \geq 1, $ we can take $ w = a_{0} + a_{1} \cdot 2
(\bmod \ 4 ) $ with $ a_{0} \in T \setminus \{ 0 \}$ and $a_{1} \in
T, $ by $ a_{0}^{2} + a_{1}(a_{0} + a_{1}) \equiv s + t \tau
(\bmod \ 8) $ and the choices of $ a_{0}, a_{1}, $ we obtain $
p \equiv 7 (\bmod 8); $  If $ v_{2}(w) < 0, \ v_{2}(z) < 0, $
by  $$ \mu w_{1}^{2}
= \mu ^{2} z_{1}^{4} + (p + q) \mu z_{1}^{2} + p q $$ with $ z_{1} =
\frac{1}{z}, \ w_{1} = \frac{w}{z^{2}} $ and $ v_{2}(w_{1}) = 0,
v_{2}(z_{1}) > 0, $ similarly as above, one can get $ p \equiv 7
(\bmod \ 8); $ If $ v_{2}(w) = 0, \ v_{2}(z) = 0, $ we can take $ w =
a_{0} + a_{1} \cdot 2 (\bmod \ 4 ), z = b_{0} (\bmod \ 2)
$ with $ a_{0}, b_{0} \in T \setminus \{ 0 \}$ and $ a_{1} \in  T, $ by
$ a_{0}^{2} + 4 a_{1}(a_{0} + a_{1}) \equiv \mu + (p + q) b_{0}^{2}
+ p \bar{\mu } b_{0}^{4}(\bmod \ 8) $ and the choices of $
a_{0}, a_{1}, b_{0}, $ we obtain $ p \equiv  1, 7 (\bmod \ 8);
$ If $ v_{2}(w) \geq 2, \ v_{2}(z) = 0, $ by $ 0 \equiv (p z^{2} +
\mu_{1}) (q z^{2} + \mu ) (\bmod \ 16 ), $ we get $ v_{2}(p z^{2}
+ \mu_{1}) \geq 3 $ or $ v_{2} ( q z^{2} + \mu ) \geq 3, $ hence $ p
\equiv 7 (\bmod \ 8); $ If $ v_{2}(w) = 1, v_{2}(z) = 0, $ let
$ w = 2 w_{0} $ with $ w_{0} \in \mathbb{O}_{K_{2}} ^{*}, $ then $ 4
\mu w_{0}^{2} \equiv (p z^{2} + \mu )(q z^{2} + \mu ) (\bmod \
32), $ it's easy to check that $p\equiv 7 (\bmod 8).$ To sum
up, if $ C^{'}(K_{2}) \neq \emptyset, $ then
$ p \equiv  17, 23 (\bmod \ 24). $
Conversely, by the above proof and Hensel lemma, it is easy to
verify that $ C_{\mu }^{'}(K_{2})\neq \emptyset, $ if
 $ p \equiv  17,23(\bmod \ 24). $  \\
\indent ii) \ By lemma 14 of \cite{Merriman}, it can be directly verified that $ C_{\mu
}^{'}(K_{\mu })\neq \emptyset, \ C_{\mu }^{'}(K_{\bar{\mu }})\neq
\emptyset $ and $ C_{\mu }^{'}(K_{p})\neq \emptyset. $ \end{proof}

\section{The computation of the Shafarevich-Tate groups}
 since $ E(K)[2] = \{O, (0, 0),
(-\varepsilon p, 0), (-\varepsilon q, 0) \}, \ \varphi(E(K)[2]) =
\{O, (0, 0) \} = E^{'}(K)[\widehat{ \varphi }], $ and $ E(K)[2]
\cong \mathbb{Z} /2 \mathbb{Z} \times \mathbb{Z} /2 \mathbb{Z} $ ( See \cite{Qiu2}), by the exact sequences
(\cite{Silverman} p.298, 314, 301 )
$$  \begin{array}{l} 0
\longrightarrow \frac{E^{'}(K )}{\varphi (E(K ))} \longrightarrow
S^{(\varphi )}(E /K ) \longrightarrow \text{TS}(E / K)[\varphi ]
\longrightarrow 0, \\
0 \longrightarrow \frac{E(K)}{\widehat{\varphi }(E^{\prime }(K ))}
\longrightarrow S^{( \widehat{\varphi })}(E^{\prime }/ K )
\longrightarrow \text{TS}(E^{\prime } / K)[\widehat{\varphi } ]
\longrightarrow 0, \\
0 \longrightarrow \frac{E^{'}(K)[\widehat{ \varphi }]}{\varphi
(E(K)[2])} \longrightarrow \frac{E^{'}(K)}{\varphi (E(K) ) }
\longrightarrow \frac{E(K)}{2E(K)} \longrightarrow
\frac{E(K)}{\widehat{ \varphi }(E^{'}(K))} \longrightarrow  0, \quad
\end{array} $$
we get
$$
\text{rank}(E(K)) + \dim_{2}(\text{TS}(E / K
)[\varphi ] ) + \dim_{2}(\text{TS}(E^{'} / K )[\widehat{\varphi }] )$$
\begin{equation}\label{eqation:basic} =\dim_{2}(S^{(\varphi )}(E / K ) ) + \dim_{2}(S^{(\widehat{\varphi
} )}(E^{'} / K ) ) - 2.\end{equation}
(A ) \ We assume that $ \varepsilon = 1. $ If $ p \equiv 3, 17
(\bmod \ 56  ), $ by \ref{prop:mod1}, we have $ S^{(\varphi
)}(E_{ + } / K ) = \{ 0 \}, $ which implies $ \varphi (E_{ + }(K )) =  E_{ +
}^{'}(K)$ and $ \text{TS}(E_{ + } / K )[\varphi ] = 0. $ Hence by
(13), we obtain \ $$ \text{rank}(E_{ + }(K)) +
\dim_{2}(\text{TS}(E_{ + }^{'} / K )[\widehat{\varphi }] ) = 1. $$
Furthermore, by the exact sequences
$$
 0 \longrightarrow \text{TS}(E_{
+ } / K )[\varphi ] \longrightarrow \text{TS}(E_{ + } / K )[2]
\longrightarrow \text{TS}(E_{ + }^{'} / K )[\widehat{\varphi }], $$
$$ 0 \longrightarrow \text{TS}(E_{ + }^{'} / K )[\widehat{\varphi }]
\longrightarrow \text{TS}(E_{ + }^{'} / K )[2] \longrightarrow
\text{TS}(E_{ + } / K )[\varphi ],
$$
we get
$ \text{TS}(E_{ + }^{'} / K )[\widehat{\varphi }]  \cong
\text{TS}(E_{ + }^{'} / K )[2], $ so $$ \text{rank}(E_{ + }(K)) +
\dim_{2}(\text{TS}(E_{ + }^{'} / K )[2] ) = 1. $$
If $ p  \equiv 45 (\bmod \ 56 ), $ then by Proposition 2.1A$_{+},$ we
have \ $ S^{(\varphi )}(E_{ + } / K ) \cong \mathbb{Z} / 2 \mathbb{Z} $ \ and \ $
S^{(\widehat{\varphi } )}(E_{ + }^{'} / K ) \cong (\mathbb{Z} / 2 \mathbb{Z} )^{2}.
$ By \ref{eqation:basic}, we obtain $$ \text{rank}(E_{ + }(K)) +
\dim_{2}(TS(E_{ + } / K )[\varphi ] ) +
\dim_{2}(\text{TS}(E_{ + }^{\prime } / K )[\widehat{\varphi }] ) = 1. $$
If $ p \equiv 31 (\bmod \ 56 ), $ by \ref{prop:mod1}, we
have \ $$  S^{(\varphi )}(E_{ + } / K ) \cong ( \mathbb{Z} / 2 \mathbb{Z} )^{2},
 S^{(\widehat{\varphi } )}(E_{ + }^{\prime } / K ) \cong (\mathbb{Z}
/ 2 \mathbb{Z} )^{3}. $$ by \ref{eqation:basic}, we obtain $$
\text{rank}(E_{ + }(K)) + \dim_{2}(TS(E_{ + } / K )[\varphi ] ) +
\dim_{2}(\text{TS}(E_{ + }^{'} / K )[\widehat{\varphi }] ) = 3. $$
This proves (A). The parts (B) $\sim $(E) can be similarly proved by
the corresponding results in Proposition 2B $\sim $ 2E. The proof of
\ref{theorem:main2} is completed.

{\bf Acknowledgements}\quad I would like thank professor Derong Qiu, who gave me this subject and much valuable advice.

\end{document}